\documentclass[a4paper,11pt]{amsart}%
\usepackage[T1]{fontenc}
\usepackage{ae}
\usepackage[latin1]{inputenc}
\usepackage{amsmath}
\usepackage{amsfonts}
\usepackage{amssymb}
\usepackage{graphicx}
\usepackage[all,cmtip]{xy}
\usepackage[dvipsnames]{xcolor}
\usepackage{hyperref}
\usepackage{tikz-cd}
\usepackage{pgfplots}%
\providecommand{\U}[1]{\protect\rule{.1in}{.1in}}
\hypersetup{
colorlinks=true,
linkcolor=blue,
filecolor=magenta,
urlcolor=cyan,
		citecolor=Green,
		        }
\newtheorem{theorem}{Theorem}[section]

\newtheorem{corollary}[theorem]{Corollary}

\newtheorem{example}[theorem]{Example}

\newtheorem*{hypothesis_A}{Hypothesis A}
\newtheorem{lemma}[theorem]{Lemma}

\newtheorem{proposition}[theorem]{Proposition}
\newtheorem{remark}[theorem]{Remark}

\makeatletter
\@namedef{subjclassname@2020}{\textup{2020} Mathematics Subject Classification}
\makeatother
\begin{document}
\title[bi-Lipschitz versus Analytic equivalence... ]{Bi-Lipschitz versus Analytic equivalence of two variable complex
quasihomogeneous function-germs}
\author[L. Câmara]{Leonardo Câmara}
\author[A. Fernandes]{Alexandre Fernandes}
\address[Leonardo Câmara]{Departamento de Matem\'atica, Universidade Federal do
Espírito Santo. Av. Fernando Ferrari, 514 - Goiabeiras, Vit\'oria - ES,
Brazil, CEP 29075-910. Email: \texttt{leonardo.camara@ufes.br}}

\address[Alexandre Fernandes]{Departamento de Matem\'atica, Universidade Federal do
Cear\'a, Rua Campus do Pici, s/n, Bloco 914, Pici, 60440-900, Fortaleza-CE,
Brazil. E-mail: \texttt{alex@mat.ufc.br}}

\thanks{The first named author has been partially supported by CAPES/PRAPG grant n%
%TCIMACRO{\U{ba} }%
%BeginExpansion
${{}^o}$
%EndExpansion
88881.964878/2024-01 and and by the Serrapilheira Institute (grant number Serra - R-2110-39576).}
\thanks{The last named author was partially supported by CNPq-Brazil grant n${{}^{o}}$ 304700/2021-5.}
\maketitle

\begin{abstract}
In this paper we address the problem of classifying complex (non-homogeneous)
quasihomogeneous polynomials in two variables under bi-Lipschitz equivalence.
We prove that pairs of such polynomials are (right) bi-Lipschitz equivalent as
function-germs at $0\in\mathbb{C}^{2}$ iff they are analytically equivalent.

\end{abstract}
\tableofcontents

%%%%%%%%%%%

%%%%%%%%%%

%%%%%%%%%%%

%%%%%%%%%%%%

%%%%%%%%%%%%%%%%%%%%%%%%%%%%% INTRODUÇÃO %%%%%%%%%%%%%%%%%%%%%%%%%%%%%%%%%%%%%

%\newpage\thispagestyle{empty}

\section{Introduction}

The classification of function-germs under analytic change of coordinates is
an ancient and very difficult problem. For instance, Whitney proved in
\cite{Wh} that each member of the family of function-germs
\begin{equation}
f_{t}(x,y)=xy(y-x)(y-tx) \label{eq:whitney}%
\end{equation}
defines a distinct equivalence class under this relation. On the other hand,
Fukuda proved in \cite{Fu} that such set of germs is discrete under
\textit{continuous change of coordinates}, i.e. under right composition with
germs of homeomorphisms, showing that under topological equivalence the family
(\ref{eq:whitney}) does not exhibit the phenomenon of continuous moduli as
seen in the analytic case. These results naturally led to the study of an
intermediary problem: the classification of germs of analytic functions under
\textit{bi-Lipschitz change of coordinates}. In this scenario, T.-C. Kuo
proved in \cite{K} that the family (\ref{eq:whitney}) is bi-Lipschitz trivial
(see also \cite[Thm. 3.3]{FR}). But this is not the case in general. In fact,
in \cite{HP1} Henry and Parusi\'{n}ski introduced a new invariant and showed
that the family
\begin{equation}
g_{t}(x,y)=y^{3}+x^{6}-3t^{2}x^{4}y;0<|t|<1/2, \label{eq:HP}%
\end{equation}
admits continuous moduli. In \cite{CamRuas2022}, Câmara and Ruas proved that
any continuous family of quasihomogeneous polynomials (see the definition
below) with fixed bi-Lipschitz type is analytically rigid, generalizing a
result previously obtained by Fernandes and Ruas in \cite{FR}.

In this paper we give a full characterization of the bi-Lipschitz and analytic
types of quasihomogeneous function-germs in two complex variables and compare
them with the topological description given by Parusi\'{n}ski in \cite{P}. In
order to state our main result, we need to recall some notions. We begin with
the classical definition of quasihomogeneous polynomial.\ Let $\varpi
=(p,q)\in\mathbb{N}^{2}$, $1\leq p\leq q$ and $\gcd(p,q)=1$. We say that
$F\in\mathbb{C}[X,Y]$ is a \emph{quasihomogeneous polynomial} with weights
$\varpi=(p,q)$ and \emph{weighted degree} $\nu\in\mathbb{N}$ if $F(\lambda
^{p}x,\lambda^{q}y)=\lambda^{\nu}F(x,y)$ for all $\lambda\in\mathbb{C}$. It is
well known that such a polynomial may also be characterized by the existence
of $a_{ij}\in\mathbb{C}$ such that $F(X,Y)=\sum_{pi+qj=\nu}a_{ij}X^{i}Y^{j}.$
%Note that the condition imposed on the indices $i$ and $j$ of satisfying the equation $pi+qj=\nu$ means that all points in the \emph{carrier }of $F$, i.e., $\Delta(F):=\{(i,j)\in
%\mathbb{N}^{2}:a_{ij}\neq0\}$, lie on the same half-line contained in the
%first quadrant whose slope is $-p/q$. From this perspective, homogeneous
%polynomials are those whose carrier have slope $-1$.

Let $\operatorname*{Diff}(\mathbb{C}^{2},0)$ denote the set of germs of
bi-analytic diffeomorphisms at $(\mathbb{C}^{2},0)$. We say that two
polynomials $F,G\in\mathbb{C}[X,Y]$ are (right) analytically equivalent as
function-germs at $0\in\mathbb{C}^{2}$ if there is $\Phi\in
\operatorname*{Diff}(\mathbb{C}^{2},0)$ such that $F=G\circ\Phi$. In a similar
way, let $\operatorname*{BiLip}(\mathbb{C}^{2},0)$ denote the set of germs of
bi-Lipschitz homeomorphisms at $(\mathbb{C}^{2},0)$. Then we say that two
polynomials $F,G\in\mathbb{C}[X,Y]$ are (right) bi-Lipschitz equivalent as
function-germs at $0\in\mathbb{C}^{2}$ if there is $\Phi\in
\operatorname*{BiLip}(\mathbb{C}^{2},0)$ such that $F=G\circ\Phi$. If $F$ and
$G$ are bi-Lipschitz equivalent, then their zeros sets $Z(F)=F^{-1}(0)$ and
$Z(G)=G^{-1}(0)$ are bi-Lipschitz equivalent as well. On the other hand, it
follows from the classical works of Burau and Zariski that the topological
type of a complex curve-germ is determined by its resolution tree (cf. e.g.
\cite{Re}). Thus, if $F$ and $G$ are (non-homogeneous) quasihomogeneous
polynomials, then it follows from the works of Câmara and Scárdua (\cite[Cor.
5.1]{CamSca2018}) that the weights of $F$ and $G$ must coincide. Notice that
in the real case scenario the situation is similar, but the technics are
distinct (cf. \cite[Prop. 5.3]{KP}).

In contrast with Whitney's family (\ref{eq:whitney}), the main result of this
paper reads as follows.

\begin{theorem}
\label{thm:main}Let $F$ and $G$ be complex (non-homogeneous) quasihomogeneous
polynomials with weights $\varpi=(p,q)$ in two complex variables. Then $F$ and
$G$ are bi-Lipschitz equivalent iff they are analytically equivalent as
function-germs at $0\in\mathbb{C}^{2}$.
\end{theorem}

It is worth noting that Theorem \ref{thm:main} confirms some previous works on
the rigidity of bi-Lipschitz equivalence compared to the analytic equivalence
of quasihomogeneous function-germs in two variables, e.g. \cite{HP1, HP2, BFP,
KP, FR2, Al, CamRuas2022}.

\section{Quasihomogeneous polynomials in two variables}

It is well known that a quasihomogeneous polynomial with weights
$\varpi=(p,q)$ and weighted degree $\nu\in\mathbb{N}$ may also be
characterized by the existence of $a_{ij}\in\mathbb{C}$ such that
\begin{equation}
F(X,Y)=\sum_{pi+qj=\nu}a_{ij}X^{i}Y^{j}. \label{eq:qhpol_loc_coord}%
\end{equation}

\-Given a quasihomogeneous polynomial $F$ with weights $\varpi=(p,q)$, (cf.
\cite{CamSca2018} and \cite{CamRuas2022}) recall that $F$ can be written in
the form
\[
F(X,Y)=c_{0}X^{m}Y^{m_{0}}{\prod\nolimits_{j=1}^{k}(}Y^{p}-\lambda_{j}X^{q}%
{)}^{{m}_{{j}}},
\]
where $c_{0}\in\mathbb{C}^{\ast}$, $\lambda_{1},\ldots,\lambda_{k}%
\in\mathbb{C}$ are pairwise distinct, and $m_{0}$, $m_{1},\ldots,m_{k}$ are
non-negative integer numbers.

\noindent\textbf{\textcolor{red}{Caveat}}: All quasihomogeneous polynomials
$F(X,Y)$ with weights $\varpi=(p,q)$ considered here are non-homogeneous in
the sense that $p<q$ and they satisfy the following conditions:

\begin{itemize}
\item In case $p=1$, we say that $F(X,Y)$ is non-homogeneous if it is not of
the form $cX^{m}(Y-\lambda X^{q})^{n}$. Let us remark that $cX^{m}(Y-\lambda
X^{q})^{n}$ is analytically equivalent to the homogeneous polynomial
$cX^{m}Y^{n}$.

\item In case $p>1$, we say that $F(X,Y)$ is non-homogeneous if it is not of
the form $cX^{m}Y^{n}$.
\end{itemize}

\noindent\textbf{Height function.} Given a quasihomogeneous polynomial
$F(X,Y)$ with weights $\varpi=(p,q)$. The one variable function $f(z)=F(1,z)$,
$z\in\mathbb{C}$, is the so-called \textit{height function of $F$}.

\section{On Lipschitz equivalence of one variable complex analytic functions
and automorphisms of $\mathbb{C}$}

A function $\phi\colon\mathbb{C}\to\mathbb{C}$ is called an
\textit{automorphism of $\mathbb{C}$} if there exist $a,b\in\mathbb{C}$,
$a\neq0$, such that $\phi(z)=az+b$ $\forall z\in\mathbb{C}$. The automorphism
$\phi$ is called a \textit{linear automorphism} if $\phi(0)=0$.

Let us start this section by addressing the following result.

\begin{proposition}
\label{prop:one_variable} Let $f,g\colon\mathbb{C}\to\mathbb{C}$ be two
non-constant analytic functions and $\phi\colon\mathbb{C}\to\mathbb{C}$ be a
Lipschitz function (not necessarily bi-Lipschitz) such that $f(z) = g\circ
\phi(z)$ $\forall z\in\mathbb{C}$, then $\phi$ is an automorphism of
$\mathbb{C}$.
\end{proposition}

\noindent\textbf{Notation}. Given two functions $f,g\colon\mathbb{C}%
\rightarrow\mathbb{C}$, $\left\vert f(z)\right\vert \lesssim\left\vert
g(z)\right\vert $ as $z\rightarrow z_{0}$ means that there exist a positive
real number $C$ and a neighborhood $U\subset\mathbb{C}$ of $z_{0}$ such that
$\left\vert f(z)\right\vert \leq C\left\vert g(z)\right\vert $ for any $z\in
U$. Furthermore, $\left\vert f(z)\right\vert \approx\left\vert g(z)\right\vert
$ as $z\rightarrow z_{0}$ means that $\left\vert f(z)\right\vert
\lesssim\left\vert g(z)\right\vert $ and $\left\vert g(z)\right\vert
\lesssim\left\vert f(z)\right\vert $ as $z\rightarrow z_{0}$.

Before we begin the proof of Proposition \ref{prop:one_variable}, let us
recall the basic concept of multiplicity and regular points.

\noindent{\textbf{Multiplicity}}. Let $h\colon\mathbb{C}\rightarrow\mathbb{C}$
be a non-constant analytic function. Given $z_{0}\in\mathbb{C}$, there exists
a unique integer $m\geq1$ such that $\left\vert h(z)-h(z_{0})\right\vert
\approx\left\vert z-z_{0}\right\vert ^{m}$ as $z\rightarrow z_{0}$. This
integer $m$ is called the \textit{multiplicity} of $h$ at $z_{0}$ and denoted
by $m(h,z_{0})$. With this nomenclature at hand, we see that the derivative
$h^{\prime}$ does not vanish at $z_{0}\in\mathbb{C}$ iff $m(h,z_{0})=1$. The
points satisfying this last condition are called \textit{regular}; the non
regular points are also known as \textit{critical} points.

\begin{lemma}
Let $\phi\colon\mathbb{C}\rightarrow\mathbb{C}$ be a Lipschitz function such
that $f(z)=g\circ\phi(z)$ $\forall z\in\mathbb{C}$, then $m(g,\phi(z))\leq
m(f,z)$ $\forall z\in\mathbb{C}$. In particular, $\phi$ maps regular points of
$f$ on regular points of $g$.
\end{lemma}

\begin{proof}
Given $z_{0}\in\mathbb{C}$, then
\begin{align*}
|z-z_{0}|^{m(f,z_{0})}  &  \approx|f(z)-f(z_{0})|\\
&  =|g(\phi(z))-g(\phi(z_{0}))|\\
&  \approx|\phi(z)-\phi(z_{0})|^{m(g,\phi(z_{0}))}\\
&  \lesssim|z-z_{0}|^{m(g,\phi(z_{0}))}\ \mbox{as}\ z\rightarrow z_{0};
\end{align*}
Therefore, $m(g,\phi(z_{0}))\leq m(f,z_{0})$.
\end{proof}

\begin{proof}
[Proof of Proposition \ref{prop:one_variable}]Let us denote by $\Sigma
_{f}\subset\mathbb{C}$ the set of the critical points of $f$. We claim that
$\phi\colon\mathbb{C}\setminus\Sigma_{f}\rightarrow\mathbb{C}$ is an analytic
function. In fact, since $\phi(z)$ is a regular point of $g$ for all
$z\in\mathbb{C}\setminus\Sigma_{f}$, it follows from the Implicit Function
Theorem that there exist neighborhoods $U\subset\mathbb{C}$ of $\phi(z)$ and
$V\subset\mathbb{C}$ of $g(\phi(z))$ such that $\left.  g\right\vert
_{U}\colon U\rightarrow V$ is bi-analytic. Since $\phi$ is a continuous
function, there is a neighborhood $W\subset\mathbb{C}$ of $z$ such that
$\phi(W)\subset U$. Then $\left.  \phi\right\vert _{W}=(\left.  g\right\vert
_{U})^{-1}\circ f$; in particular, $\left.  \phi\right\vert _{\mathbb{C}
\setminus\Sigma_{f}}$ is an analytic function.

Once we have proved that $\phi\colon\mathbb{C}\setminus\Sigma_{f}
\rightarrow\mathbb{C}$ is an analytic function and $\Sigma_{f}\subset
\mathbb{C}$ is a discrete subset, then it follows from Riemann's Removable
Singularity Theorem that $\phi\colon\mathbb{C}\rightarrow\mathbb{C}$ is an
analytic function.

Finally, since $\phi\colon\mathbb{C}\rightarrow\mathbb{C}$ is a globally
Lipschitz analytic function, it comes from Liouville's Theorem that $\phi$ is
an affine function, i.e. there exists $a,b\in\mathbb{C}$ such that
$\phi(z)=az+b$ $\forall z\in\mathbb{C}$; since $f$ is non-constant, then
$a\neq0$.
\end{proof}

At this point, it is natural to consider the following equivalence relation on
the sets of one variable complex functions: we say that two functions
$f,g\colon\mathbb{C}\rightarrow\mathbb{C}$ are \emph{equivalent under
automorphisms of $\mathbb{C}$} if there exist a non-zero constant $\sigma
\in\mathbb{C}^{*}$ and an automorphism $\phi\colon\mathbb{C}\rightarrow
\mathbb{C}$ such that $\sigma f=g\circ\phi.$ In this case, if the automorphism
$\phi$ is linear we say $f$ and $g$ are \emph{equivalent under linear
automorphisms of $\mathbb{C}$}.

In the following, we are going to encode the equivalence classes of
polynomials by using their multisets of zeros. First, let us recall the notion
of multisets of zeros of one complex variable polynomials.

\noindent\textbf{Multiset of zeros}. Given $f(z) = a_{n} z^{n} + \cdots+
a_{0}$ (with $a_{n} \neq0$), factorize it as
\[
f(z) = a_{n} (z - r_{1})^{m_{1}} (z - r_{2})^{m_{2}} \cdots(z - r_{k})^{m_{k}%
},
\]
where $r_{1}, r_{2}, \ldots, r_{k}$ are the distinct roots and $m_{i}$ is the
multiplicity of $r_{i}$.

Thus, given a non-zero complex number $c$, the roots of $cf(z)$ and their
respective multiplicities remain exactly the same as $f$, only the leading
coefficient changes to $c a_{n}$. So, the multiset of zeros
\[
\{r_{j} : m_{j}\}_{j=1}^{k}=\{r_{1} : m_{1}, r_{2} : m_{2}, \ldots, r_{k} :
m_{k}\}
\]
is invariant across the equivalence class.

Actually, we have shown that $g(z)$ is equal to $cf(z)$ for some
$c\in\mathbb{C}^{*}$ iff $g(z)$ and $f(z)$ have the same multiset of zeros.
Hence, we easily arrive at the following result.

\begin{proposition}
Let $f(z)$ be a non-constant complex polynomial with multiset of zeros
$\{r_{j} : m_{j}\}_{j=1}^{k}$. Then the following statements hold true:

\begin{enumerate}
\item $g(z)$ and $f(z)$ are equivalent under automorphisms of $\mathbb{C}$ if
and only if $\exists a\neq0,b$ in $\mathbb{C}$ such that the multisets of
zeros of $g(z)$ is equal to $\{ar_{j}+b : m_{j}\}_{j=1}^{k}$. In particular,
$g(z)$ and $f(z)$ have the same degree.

\item $g(z)$ and $f(z)$ are equivalent under linear automorphisms of
$\mathbb{C}$ if and only if $\exists a\neq0$ in $\mathbb{C}$ such that the
multisets of zeros of $g(z)$ is equal to $\{ar_{j} : m_{j}\}_{j=1}^{k}$. In
particular, $g(z)$ and $f(z)$ have the same degree
\end{enumerate}
\end{proposition}

\section{Weighted blow-up of a function-germ\label{section:blow-up_section}}

Here we introduce one of the main ingredients for the proof of Theorem
\ref{thm:main}, the notion of weighted blow-up of a Lipschitz function
vanishing at the origin. Somehow, we emulate a blow-up of Lipschitz mappings
at the origin presented in \cite{S}.

Let $\psi\colon\mathbb{C}^{2}\rightarrow\mathbb{C}$ be a $c$-Lipschitz
function, i.e. a Lipschitz function with Lipschitz constant $c>0$ such that
$\psi(0,0)=0$. Let $\varpi=(p,q)$ be a pair of relatively prime positive
integer numbers such that $p<q$. Given a sequence of positive real numbers
$\{t_{n}\}_{n\in\mathbb{N}}$ such that $t_{n}\rightarrow0$ and a compact disc
$D\subset\mathbb{C}$, we define a sequence of functions $\phi_{n}\colon
D\rightarrow\mathbb{C}$ by $\phi_{n}(z):=\frac{\psi(t_{n}^{p},t_{n}^{q}%
z)}{t_{n}^{q}}.$

Since $\psi$ is a $c$-Lipschitz function, it is clear that, for each
$n\in\mathbb{N}$, $\phi_{n}$ is also a $c$-Lipschitz function.

\begin{hypothesis_A}
\noindent$\exists\ z_{0}\in\mathbb{C}$ such that $\psi(s^{p},s^{q}%
z_{0})=O(s^{q})$ as $\ s\rightarrow0^{+}$.
\end{hypothesis_A}

Under the previous hypotheses, this sequence of functions becomes uniformly
bounded. Indeed, let us assume that $\exists M>0$ such that $|\phi_{n}%
(z_{0})|\leq M$ $\forall n$. Then, given $z\in D$, we obtain
\begin{align*}
|\phi_{n}(z)|  &  = |\phi_{n}(z)-\phi_{n}(z_{0}) + \phi_{n}(z_{0})|\\
&  \leq|\phi_{n}(z)-\phi_{n}(z_{0})| + |\phi_{n}(z_{0})|\\
&  \leq c|z-z_{0}| + M.
\end{align*}
Therefore, since $D$ is compact, $\exists K>0$ such that $|\phi_{n}(z)|\leq K$
$\forall n\in\mathbb{N}$ and $\forall z\in D$.

At this point, we can apply Arzela-Ascoli's Theorem to ensure that there
exists a subsequence of $\{\phi_{n}\}$ that converges uniformly to a function
$\phi\colon D\rightarrow\mathbb{C}$, which is a $c$-Lipschitz function as well.

Based on the observations made in the above paragraphs, we will construct a
weighted blow-up of the function $\psi(x,y)$ at $(0,0)$ as a $c$-Lipschitz
function $\phi\colon\mathbb{C}\rightarrow\mathbb{C}$ as follows: There exists
a sequence $\{s_{k}\}_{k\in\mathbb{N}}$, $s_{k}\to0$, such that for each
compact subset $D\subset\mathbb{C}$, the sequence of functions
\[
z\in D\mapsto\frac{\psi(s_{k}^{p},s_{k}^{q}z)}{s_{k}^{q}}%
\]
converges uniformly to $\left.  \phi\right\vert _{D}\colon D\rightarrow
\mathbb{K}$. Actually, we start by taking $\{t_{n}\}_{n\in F_{1}}$ a
subsequence of $\{t_{n}\}_{n\in\mathbb{N}}$ such that the sequence of
functions
\[
z\mapsto\frac{\psi(t_{n}^{p},t_{n}^{q}z)}{t_{n}^{q}};\ n\in F_{1}%
\]
converges uniformly to a $c$-Lipschitz function ${}^{\left(  1\right)  }%
\!\phi\colon D_{1}\rightarrow\mathbb{C}$, where $D_{k}=\{z\in\mathbb{C}%
\ :\ \left\vert z\right\vert \leq k\}$. By induction on $k\geq1$, one can
construct subsequences $\{t_{n}\}_{n\in F_{k}}$ of $\{t_{n}\}_{n\in\mathbb{N}
}$ such that $\cdots\subset F_{k+1}\subset F_{k}\subset\cdots\subset
\mathbb{N}$ and the sequence of functions
\[
z\in D_{k+1}\mapsto\frac{\psi(t_{n}^{p},t_{n}^{q}z)}{t_{n}^{q}};\ n\in F_{k+1}%
\]
converges uniformly to a $c$-Lipschitz function ${}^{(k+1)}\!\phi\colon
D_{k+1}\rightarrow\mathbb{C}$ such that ${}^{(k+1)}\!\left.  \phi\right\vert
_{D_{k}}={}^{\left(  k\right)  }\!\phi$. Now consider the sequence
$\{s_{k}\}_{k\in\mathbb{N}}$ defined for each $k$ by $s_{k}=t_{n_{k}}\in
F_{k}$; $n_{1}<n_{2}<\cdots<n_{k}<\cdots$. Then we define $\phi\colon
\mathbb{C}\rightarrow\mathbb{C}$ by
\[
z\in\mathbb{C}\mapsto\ \phi(z)=\lim_{k\rightarrow\infty}\frac{\psi(s_{k}%
^{p},s_{k}^{q}z)}{s_{k}^{q}}.
\]
This function $\phi$ is called a $\varpi$\emph{-weighted blow-up of }%
$\psi(x,y)$\emph{ at }$(0,0)$ and the sequence $\{s_{k}\}_{k\in\mathbb{N}}$ is
known as an \emph{effective sequence} associated to $\phi$.

\begin{remark}
\label{rem:local_and_bi-Lip} For the sake of simplicity, the notion of
$\varpi$-weighted blow-up was introduced for $c$-Lipschitz functions
$\psi\colon\mathbb{C}^{2}\rightarrow\mathbb{C}$ such that $\psi(0,0)=0$. But
as a matter of fact, since $(s_{k}^{p},s_{k}^{q}z)\rightarrow(0,0)$ as
$s_{k}\rightarrow0$, the same concept may be introduced for function-germs
$\psi\colon\left(  \mathbb{C}^{2},0\right)  \rightarrow\left(  \mathbb{C}%
,0\right)  $ as well.
\end{remark}

\section{Height functions and bi-Lipschitz equivalence}

The main result of this section is the following proposition that shows how
the height functions of two germs of (non-homogeneous) quasihomogeneous
functions that are Lipschitz equivalent are related.

But first let us recall from the classical works of Burau and Zariski that the
weights $\varpi=(p,q)$ of a quasihomogeneous curve are topologically invariant
(cf. \cite{Bu}, \cite{Za}). In particular, any two topologically equivalent
quasihomogeneous polynomials must share the same weights. Therefore, from now
on we shall only consider quasihomogeneous polynomials satisfying this property.

\begin{proposition}
\label{prop:main proposition} Let $F,G\in\mathbb{C}[X,Y]$ be two
(non-homogeneous) quasihomogeneous polynomials with the same weights
$\varpi=(p,q)$ and weighted degrees $\nu(F)$ and $\nu(G)$, respectively. If
$F$ and $G$ are (right) bi-Lipschitz equivalent as function-germs at the
origin then $\nu(F)=\nu(G)$ and their height functions $f$ and $g$ are
equivalent under automorphisms of $\mathbb{C}$.
\end{proposition}

Before starting the proof of Proposition \ref{prop:main proposition}, let us
apply this result to recover the existence of moduli for the right
bi-Lipschitz equivalence of function-germs in the Henry-Parusi\'nski's example below.

\begin{example}
Let $F_{t}(X,Y)=Y^{3}+tYX^{4}+X^{6}$ be a 1-parameter family of
quasihomogeneous polynomials with weights $\varpi=(1,2)$ and weighted degree
$6$. If $t^{3}=s^{3}$, by choosing $a\in\mathbb{C}$ such that $t=as$ and
$a^{3}=1$, we get $F_{t}(X,Y)=F_{s}(X,aY)$, hence $F_{t}$ and $F_{s}$ are
right analytically equivalent. In the seminal paper \cite{HP1}, Henry and
Parusi\'{n}ski used the family $F_{t}(X,Y)$ to show the first example of the
existence of moduli for the right bi-Lipschitz equivalence of analytic
function-germs, more precisely, they proved that $F_{t}$ and $F_{s}$ are right
bi-Lipschitz equivalent as function-germs at $0\in\mathbb{C}^{2}$ only if
$t^{3}=s^{3}$. Let us show how to recover this result of Henry-Parusi\'{n}ski.
Let us suppose that $F_{t}$ and $F_{s}$ are right bi-Lipschitz equivalent as
function-germs at $0\in\mathbb{C}^{2}$. By Proposition 5.1, the respective
height functions $f_{t}(z)=F_{t}(1,z)$ and $f_{s}(z)=F_{s}(1,z)$ are
equivalent under automorphisms of $\mathbb{C}$, i.e. $cf_{t}(z)=f_{s}(az+b)$
for some $a,c\in\mathbb{C}^{\ast}$ and $b\in\mathbb{C}$, hence the polynomial
identity
\begin{align*}
cz^{3}+ctz+c  &  =(az+b)^{3}+s(az+b)+1\\
&  =a^{3}z^{3}+3a^{2}bz^{2}+(3ab^{2}+sa)z+b^{3}+sb+1
\end{align*}
implies that: $c=1$, $b=0$, $a^{3}=1$ and $t=as$, i.e. $t^{3}=s^{3}$.
\end{example}

\begin{proof}
[Proof of Proposition \ref{prop:main proposition}]Without loss of generality,
let us assume that $\nu(F)\leq\nu(G)$. Let $\Phi:=\left(  \Phi_{1},\Phi
_{2}\right)  \in\operatorname*{BiLip}\left(  \mathbb{C}^{2},0\right)  $ be
such that $F(X,Y)=G\circ\Phi(X,Y)$.

\noindent\textbf{Claim 1}. The function $\Phi_{2}$ satisfies Hypothesis A of
Section \ref{section:blow-up_section}.

We are going to prove Claim 1 by mimicking the arguments used in the proof of
\cite[Lem. 3.1]{Al}. In fact, let us write the quasihomogeneous polynomial
$G(X,Y)$ with weights $\varpi=(p,q)$ in the form
\[
G(X,Y)=\sum_{j=0}^{m}c_{j}X^{d-qj}Y^{pj},\quad c_{j}\in\mathbb{C},
\]
where $pd=\nu(G)$ and $mq=d$. Let $z_{0}\in\mathbb{C}$ be a zero of the height
function $f$ of $F$. Then, for each sufficiently small $s\in\mathbb{C}^{\ast}%
$, we have
\begin{equation}
0=f(z_{0})=G(\frac{\Phi_{1}(s^{p},s^{q}z_{0})}{s^{p}},\frac{\Phi_{2}%
(s^{p},s^{q}z_{0})}{s^{q}}) \label{eq: phi_2}%
\end{equation}
By applying Cauchy's bound on the roots of the polynomial equation
(\ref{eq: phi_2}), we obtain
\begin{equation}
\left\vert \frac{\Phi_{2}(s^{p},s^{q}z_{0})}{s^{q}}\right\vert \leq
1+\max\left\{  \frac{|c_{m-1}|}{c_{m}}|\tilde{x}^{q}|,\dots,\frac{|c_{1}%
|}{c_{m}}|\tilde{x}^{q(m-1)}|,\frac{|c_{0}|}{c_{m}}|\tilde{x}^{qm}|\right\}  ,
\label{ineq: phi_2}%
\end{equation}
where $\tilde{x}=\frac{\Phi_{1}(s^{p},s^{q}z_{0})}{s^{p}}$. Since $\Phi_{1}$
is a Lipschitz function, then $|\tilde{x}|$ is bounded. Hence, it follows from
(\ref{ineq: phi_2}) that $\tilde{x}=\frac{\Phi_{2}(s^{p},s^{q}z_{0})}{s^{q}}$
is bounded as well. Thus, Claim 1 is proven.

As $\Phi_{2}$ satisfies Hypothesis A of Section \ref{section:blow-up_section},
we can make use of a $\varpi$-weighted blow-up $\phi$ of $\Phi_{2}$ at the
origin with its respective effective sequence of real positive numbers
$\{s_{n}\}_{n\in\mathbb{N}}.$

Besides, since $\left\vert \Phi_{1}(s_{n}^{p},0)\right\vert \lesssim s_{n}%
^{p}$ as $s_{n}\rightarrow0$, up to a subsequence (if necessary), one may
suppose that $\frac{\Phi_{1}(s_{n}^{p},0)}{s_{n}^{p}}\rightarrow\mu$.

\noindent\textbf{Claim 2.} $\mu\neq0$.

In fact, since $\Phi$ is bi-Lipschitz, we know that $\left\vert \Phi(s_{n}%
^{p},0)\right\vert \approx s_{n}^{p}$ as $s_{n}\rightarrow0$. On the other
hand, we also know that $\lim\frac{\Phi_{2}(s_{n}^{p},0)}{s_{n}^{q}}=\phi(0)$,
i.e. $\left\vert \Phi_{2}(s_{n}^{p},0)\right\vert \lesssim s_{n}^{q}$ as
$s_{n}\rightarrow0$; hence, $\left\vert \Phi_{1}(s_{n}^{p},0)\right\vert
\approx s_{n}^{p}$ as $s_{n}\rightarrow0$ and $\mu\neq0$.

\noindent\textbf{Claim 3.} Given $z\in\mathbb{C}$, $\frac{\Phi_{1}(s_{n}%
^{p},s_{n}^{q}z)}{s_{n}^{p}}\to\mu$.

In fact, the proof of this claim is a direct consequence of the Lipschitz
property of the function $\Phi_{1}$ as one can see bellow:
\[
\left\vert \frac{\Phi_{1}(s_{n}^{p},s_{n}^{q}z)}{s_{n}^{p}}-\frac{\Phi
_{1}(s_{n}^{p},0)}{s_{n}^{p}}\right\vert \lesssim s_{n}^{q-p}%
\ \mbox{as}\ s_{n}\rightarrow0.
\]

Besides, notice that
\begin{align*}
s_{n}^{\nu(F)}F(1,z)  &  =F(s_{n}^{p},s_{n}^{q}z)\\
&  =G(\Phi_{1}(s_{n}^{p},s_{n}^{q}z),\Phi_{2}(s_{n}^{p},s_{n}^{q}z))\\
&  =s_{n}^{\nu(G)}G(s_{n}^{-p}\cdot\Phi_{1}(s_{n}^{p},s_{n}^{q}z),s_{n}^{-q}
\Phi_{2}(s_{n}^{p},s_{n}^{q}z)).
\end{align*}
Now, for the sake of simplicity, let us write $\mu$ as $\mu=\sigma^{-p}$. From
the above equations, we have the following consequences. Firstly, since
\[
\lim_{n\rightarrow\infty}G(s_{n}^{-p}\cdot\Phi_{1}(s_{n}^{p},s_{n}^{q}%
z),s_{n}^{-q}\Phi_{2}(s_{n}^{p},s_{n}^{q}z))=G(\sigma^{-p},\phi(z)),
\]
$\nu(F)\leq\nu(G)$, and $F(1,z)$ is not the zero function, we must have
$\nu(F)=\nu(G)=\nu$. Secondly,
\begin{align*}
\sigma^{\nu}\cdot f(z)  &  =\sigma^{\nu}\cdot\lim_{n\rightarrow\infty}%
G(s_{n}^{-p}\cdot\Phi_{1}(s_{n}^{p},s_{n}^{q}z),s_{n}^{-q}\Phi_{2}(s_{n}%
^{p},s_{n}^{q}z))\\
&  =\sigma^{\nu}\cdot G(\sigma^{-p},\phi(z))\\
&  =G(1,\varphi(z))=g(\varphi(z)),
\end{align*}
where $\varphi(z)=\sigma^{q}\cdot\phi(z)$ is a Lipschitz function. The result
then follows by Proposition \ref{prop:one_variable}.
\end{proof}

\begin{lemma}
\label{lem:Lip_analytic} Let $F,G\in\mathbb{C}[X,Y]$ be two (non-homogeneous)
quasihomogeneous polynomials with weights $\varpi=(p,q)$. Let $f$ and $g$ be
the height functions of $F$\ and $G$, respectively. Suppose that $f$ and $g$
are equivalent under automorphisms of $\mathbb{C}$. If $p>1$ then $f$ and $g$
are equivalent under linear automorphisms of $\mathbb{C}$.
\end{lemma}

\begin{proof}
Let us consider the assumptions of the lemma and $p>1$. We see the roots of
both $f$ and $g$ are invariant by rotations by $p$th roots of unity. But
$\phi$ commutes with such rotations iff $\phi(0)=0$. The result then follows.
\end{proof}

\begin{corollary}
\label{cor:multiplicity_invariance1} Let $F,G\in\mathbb{C}[X,Y]$ be two
(non-homogeneous) quasihomogeneous polynomials with weights $1=p<q$. Let us
write:
\[
F(X,Y)=c_{0}X^{m}{\textstyle\prod\nolimits_{j=0}^{k}} (Y-\lambda_{j}%
X^{q})^{m_{j}}, \ \mbox{with} \ m_{0}\leq\cdots\leq m_{k},
\]
some $\lambda_{j}=0$ and $\lambda_{0},\dots,\widehat{\lambda_{j}}%
,\dots,\lambda_{k}\in\mathbb{C}^{*}$, and
\[
G(X,Y)=d_{0}X^{n}{\textstyle\prod\nolimits_{j=0}^{\ell}} (Y-\mu_{j}%
X^{q})^{n_{j}} \ \mbox{with} \ n_{0}\leq\cdots\leq n_{\ell},
\]
some $\mu_{i}=0$ and $\mu_{0},\dots,\widehat{\mu_{i}},\dots,\mu_{\ell}%
\in\mathbb{C}^{*}$.

If $F$ and $G$ are right bi-Lipschitz equivalent as function-germs at
$0\in\mathbb{C}^{2}$ then $k=\ell$, $m=n$, $m_{j}=n_{j}$ $\forall
\ j=0,1,\dots,k$, and there are $a\in\mathbb{C}^{*}$, $b\in\mathbb{C}$ such
that $a\lambda_{j}+b=\mu_{j}$ $\forall\ j=1,\dots,k$.
\end{corollary}

\begin{proof}
Let $f(z)$ and $g(z)$ be the height functions of $F$ and $G$, respectively. We
have
\[
f(z)= c_{0}{\textstyle\prod\nolimits_{j=0}^{k}} (z-\lambda_{j})^{m_{j}}
\ \mbox{and} \ g(z)=d_{0}{\textstyle\prod\nolimits_{j=0}^{\ell}} (z-\mu
_{j})^{n_{j}}.
\]
Since $F$ and $G$ are right bi-Lipschitz equivalent, by Proposition
\ref{prop:main proposition} we have that $f(z)$ and $g(z)$ are equivalent
under automorphisms of $\mathbb{C}$, let say $cf(z)=g(\phi(z))$ such that
$c\neq0$ and $\phi$ is an automorphism of $\mathbb{C}$. It follows that $f(z)$
and $g(\phi(z))$ have the same multiset of zeros. In particular, $k=\ell$,
$m_{j}=n_{j}$ and $\phi(\lambda_{j})=\mu_{j}$ $\forall\ j=0,\dots,k$.

Finally, since $F$ and $G$ are right bi-Lipschitz equivalent as function-germs
at $0\in\mathbb{C}^{2}$, we know that the order of $F$ and $G$ at $0$ are the
same, $\mathrm{ord}_{0}F=\mathrm{ord}_{0}G$ (see, for instance, \cite{P, RT}).
On the other hand,
\[
\mathrm{ord}_{0}F= m + \sum_{j=0}^{k} m_{j} \ \mbox{and} \ \mathrm{ord}_{0}G=
n + \sum_{j=0}^{k} m_{j}.
\]
We conclude that $m=n$ as well.
\end{proof}

\begin{corollary}
\label{cor:multiplicity_invariance2} Let $F,G\in\mathbb{C}[X,Y]$ be two
(non-homogeneous) quasihomogeneous polynomials with weights $1<p<q$. Let us
write
\[
F(X,Y)=c_{0}X^{m}Y^{m_{0}}{\textstyle\prod\nolimits_{j=1}^{k}} (Y^{p}%
-\lambda_{j}X^{q})^{m_{j}}, \ \mbox{with} \ \lambda_{j}\in\mathbb{C}^{*},
\ m_{1}\leq\cdots\leq m_{k},
\]
and
\[
G(X,Y)=d_{0}X^{n}Y^{n_{0}}{\textstyle\prod\nolimits_{j=1}^{\ell}} (Y^{p}%
-\mu_{j}X^{q})^{n_{j}} \ \mbox{with} \ \mu_{j}\in\mathbb{C}^{*}, \ n_{1}%
\leq\cdots\leq n_{\ell}.
\]
If $F$ and $G$ are right bi-Lipschitz equivalent as function-germs at
$0\in\mathbb{C}^{2}$ then $k=\ell$, $m=n$, $m_{j}=n_{j}$ $\forall
\ j=0,1,\dots,k$, and there is $a\in\mathbb{C}^{*}$ such that $a\lambda
_{j}=\mu_{j}$ $\forall\ j=1,\dots,k$.
\end{corollary}

\begin{proof}
Let $f(z)$ and $g(z)$ be the height functions of $F$ and $G$, respectively.
Then we have
\[
f(z)= c_{0}{\textstyle\prod\nolimits_{j=1}^{k}} (z^{p}-\lambda_{j})^{m_{j}}
\ \mbox{and} \ g(z)=d_{0}{\textstyle\prod\nolimits_{j=1}^{\ell}} (z^{p}%
-\mu_{j})^{n_{j}}.
\]
Since $F$ and $G$ are right bi-Lipschitz equivalent, by Proposition
\ref{prop:main proposition} we have that $f(z)$ and $g(z)$ are equivalent
under linear automorphisms of $\mathbb{C}$, let say $cf(z)=g(\phi(z))$ such
that $c\neq0$ and $\phi$ is an automorphism of $\mathbb{C}$; $\phi(0)=0$. It
follows that $f(z)$ and $g(\phi(z))$ have the same multiset of zeros, i.e.
\[
\{0:n_{0},\mu_{1}:n_{1},\dots,\mu_{\ell}:n_{\ell}\}= \{0:m_{0},\phi
(\lambda_{1}):m_{1},\dots,\phi(\lambda_{k}):m_{k}\}.
\]
In particular, $k=\ell$, $m_{j}=n_{j}$, and $\mu_{j}=\phi(\lambda_{j})$
$\forall\ j=0,\dots,k$.

Finally, since $F$ and $G$ are right bi-Lipschitz equivalent as function-germs
at $0\in\mathbb{C}^{2}$, we know that the order of $F$ and $G$ at $0$ are the
same, $\mathrm{ord}_{0}F=\mathrm{ord}_{0}G$ (see, for instance, \cite{P, RT}).
On the other hand,
\[
\mathrm{ord}_{0}F= m + m_{0} + p\sum_{j=1}^{k} m_{j} \ \mbox{and}
\ \mathrm{ord}_{0}G= n + m_{0} + p\sum_{j=1}^{k} m_{j}.
\]
Hence, $m=n$ as well.
\end{proof}

\section{Bi-Lipschitz equivalence vs. Analytic equivalence}

At this point, we are ready to present a proof of Theorem \ref{thm:main}.
Actually, we prove a result stronger than Theorem \ref{thm:main}, as one can
see bellow.

\begin{theorem}
\label{thm:main2} Let $F,G\in\mathbb{C}[X,Y]$ be (non-homogeneous)
quasihomogeneous polynomials with the same weights $\varpi=(p,q)$. Suppose
that $F$ and $G$ are (right) bi-Lipschitz equivalent as function-germs at
$0\in\mathbb{C}^{2}$. Then we have the following alternatives:

\begin{enumerate}
\item[a)] If $p=1$ then there are $\alpha,\beta\in\mathbb{C}^{\ast}$ and
$\gamma\in\mathbb{C}$ such that $F(X,Y)=G(\alpha X,\beta Y-\gamma X^{q})$.

\item[b)] If $p>1$ then there are $\alpha,\beta\in\mathbb{C}^{*}$ such that
$F(X,Y)=G(\alpha X,\beta Y)$.
\end{enumerate}
\end{theorem}

\begin{remark}
Let us point out that Theorem \ref{thm:main2} is analogous to \cite[Prop.
5.4]{KP} for $C^{1}$ right equivalence of quasihomogeneous polynomials in two
real variables. We believe that our result does not depend on that one because
we are considering local bi-Lipschitz right equivalence (which is weaker than
$C^{1}$ right equivalence) and our result is in the setting of complex
variables while \cite[Prop. 5.4]{KP} is in the context of real variables.
Besides, for (real) bi-Lipschitz equivalence the analogous result does not
hold (cf. \cite[Exs. 6.1 and 6.2.]{KP}).
\end{remark}

\begin{proof}
[Proof of Theorem \ref{thm:main2}]Since $F$ and $G$ are (right) bi-Lipschitz
equivalent as function-germs at $0\in\mathbb{C}^{2}$, it follows from
Proposition \ref{prop:main proposition} that $F$ and $G$ have the same
weighted degree, let us say $\nu$.

\noindent\textbf{Case $p=1$}. In this case, by Corollary
\ref{cor:multiplicity_invariance1}, one can write
\[
F(X,Y)=c_{0}X^{m}{\textstyle\prod\nolimits_{j=0}^{k}}(Y-\lambda_{j}%
X^{q})^{m_{j}}%
\]
and
\[
G(X,Y)=d_{0}X^{m}{\textstyle\prod\nolimits_{j=0}^{k}}(Y-\mu_{j}X^{q})^{m_{j}%
}.
\]
Corollary \ref{cor:multiplicity_invariance1} also says that $\exists
a\in\mathbb{C}^{\ast}$, $\exists b\in\mathbb{C}$ such that $a\mu_{j}%
+b=\lambda_{j}$ $\forall\ j=0,\dots,k$. Then
\begin{align*}
F(X,Y)  &  =d_{0}X^{m}{\textstyle\prod\nolimits_{j=0}^{k}}(Y-\lambda_{j}%
X^{q})^{m_{j}}\\
&  =d_{0}X^{m}{\textstyle\prod\nolimits_{j=0}^{k}}Y-(a\mu_{j}+b)X^{q})^{m_{j}%
}\\
&  =d_{0}X^{m}{\textstyle\prod\nolimits_{j=0}^{k}}((Y-bX^{q})-a\mu_{j}%
X^{q})^{m_{j}}.
\end{align*}
By taking $\beta\in\mathbb{C}^{\ast}$ such that $\beta^{\sum_{j=0}^{k}m_{j}%
}=d_{0}/c_{0}$, we obtain
\[
F(X,Y)=c_{0}X^{m}{\textstyle\prod\nolimits_{j=0}^{k}}((\beta Y-b\beta
X^{q})-a\beta\mu_{j}X^{q})^{m_{j}}.
\]
By taking $\alpha\in\mathbb{C}^{\ast}$ such that $\alpha^{q}=a\beta$ and
$\gamma=b\beta$, we obtain
\begin{align*}
F(X,Y)  &  =c_{0}X^{m}{\textstyle\prod\nolimits_{j=1}^{k}}((\beta Y-\gamma
X^{q})-\mu_{j}(\alpha X)^{q})^{m_{j}}\\
&  =\alpha^{-m}G(\alpha X,\beta Y-\gamma X^{q}).
\end{align*}
Let us consider $s\in\mathbb{C}^{\ast}$ such that $s^{\nu}=\alpha^{-m}$ and
let us put $\alpha^{\prime}=s\alpha$, $\beta^{\prime}=s^{q}\beta$, and
$\gamma^{\prime}=s^{q}\gamma$. Thus, $F(X,Y)=G(\alpha^{\prime}X,\beta^{\prime
}Y-\gamma^{\prime}X^{q}).$

\noindent\textbf{Case $p>1$}. In this case, by Corollary
\ref{cor:multiplicity_invariance2}, one can write
\[
F(X,Y)=c_{0}X^{m}Y^{m_{0}}{\textstyle\prod\nolimits_{j=0}^{k}}(Y^{p}%
-\lambda_{j}X^{q})^{m_{j}}%
\]
and
\[
G(X,Y)=d_{0}X^{m}Y^{m_{0}}{\textstyle\prod\nolimits_{j=0}^{k}}(Y^{p}-\mu
_{j}X^{q})^{m_{j}}.
\]
Corollary \ref{cor:multiplicity_invariance2} also says that $\exists
a\in\mathbb{C}^{\ast}$ such that $a\mu_{j}=\lambda_{j}$ $\forall\ j=1,\dots
,k$. Then%
\begin{align*}
F(X,Y)  &  =c_{0}X^{m}Y^{m_{0}}{\textstyle\prod\nolimits_{j=1}^{k}}%
(Y^{p}-\lambda_{j}X^{q})^{m_{j}}\\
&  =c_{0}X^{m}Y^{m_{0}}{\textstyle\prod\nolimits_{j=1}^{k}}(Y^{p}-a\mu
_{j}X^{q})^{m_{j}}.
\end{align*}
Let $b\in\mathbb{C}^{\ast}$ satisfy $\beta^{m_{0}+p\sum_{j=1}^{k}m_{j}}%
=c_{0}/d_{0}$. Then
\[
F(X,Y)=d_{0}X^{m}(\beta Y)^{m_{0}}{\textstyle\prod\nolimits_{j=1}^{k}}((\beta
Y)^{p}-a\beta^{p}\mu_{j}X^{q})^{m_{j}}.
\]
By taking $\alpha\in\mathbb{C}^{\ast}$ such that $\alpha^{q}=a\beta^{p}$, we
obtain
\begin{align*}
F(X,Y)  &  =d_{0}X^{m}(\beta Y)^{m_{0}}{\textstyle\prod\nolimits_{j=1}^{k}%
}((\beta Y)^{p}-\mu_{j}(\alpha X)^{q})^{m_{j}}\\
&  =\alpha^{-m}G(\alpha X,\beta Y).
\end{align*}
Let us consider $s\in\mathbb{C}^{\ast}$ such that $s^{\nu}=\alpha^{-m}$, and
let us put $\alpha^{\prime}=s^{p}\alpha$, $\beta^{\prime}=s^{q}\beta$. Thus,
$F(X,Y)=G(\alpha^{\prime}X,\beta^{\prime}Y).$
\end{proof}

As we mentioned, Theorem \ref{thm:main} follows as a corollary of Theorem
\ref{thm:main2}.

\begin{corollary}
[Theorem \ref{thm:main}]Let $F$ and $G$ be complex (non-homogeneous)
quasihomogeneous polynomials in two complex variables. Then $F$ and $G$ are
bi-Lipschitz equivalent iff they are analytically equivalent as function-germs
at $0\in\mathbb{C}^{2}$.
\end{corollary}

\section{Moduli space}

Since the bi-Lipschitz and analytic type of (non-homogeneous) quasihomogeneous
polynomials in $\mathbb{C}[X,Y]$ coincide (cf. Theorem \ref{thm:main}), now we
shall describe the analytic moduli space of (non necessarily reduced) germs of
quasihomogeneous function-germs in the same spirit as performed in
\cite{CamSca2018} and \cite{CamRuas2022} for the reduced case.

Recall that a quasihomogeneous polynomial with weights $\varpi=(1,q)$, $q>1$,
is (uniquely) determined by
\begin{equation}
F(X,Y)=X^{m}{\textstyle\prod\nolimits_{i_{1}=1}^{k_{1}}}(Y-\lambda_{1,i_{1}%
}X^{q})^{m_{1}}\cdots{\textstyle\prod\nolimits_{i_{n}=1}^{k_{n}}}%
(Y-\lambda_{n,i_{n}}X^{q})^{m_{n}}, \label{eq:norm_form_p=1}%
\end{equation}
where $\lambda_{1,1},\cdots,\lambda_{1,k_{1}},\cdots,\lambda_{n,1}%
,\cdots,\lambda_{n,k_{n}}\in\mathbb{C}$ are pairwise distinct, and
$0<m_{1}<\dots<m_{n}$ are non-negative integer numbers. Analogously, a
quasihomogeneous polynomial with weights $\varpi=(p,q)$, $1<p<q$, is
(uniquely) determined by
\begin{equation}
F(X,Y)=X^{m}Y^{m_{0}}{\textstyle\prod\nolimits_{i_{1}=1}^{k_{1}}}%
(Y^{p}-\lambda_{1,i_{1}}X^{q})^{m_{1}}\cdots{\textstyle\prod\nolimits_{i_{n}%
=1}^{k_{n}}}(Y^{p}-\lambda_{n,i_{n}}X^{q})^{m_{n}}, \label{eq:norm_form_p>1}%
\end{equation}
where $\lambda_{1,1},\cdots,\lambda_{1,k_{1}},\cdots,\lambda_{n,1}%
,\cdots,\lambda_{n,k_{n}}\in\mathbb{C}^{\ast}$ are pairwise distinct, and
$0<m_{0}<m_{1}<\dots<m_{n}$ are non-negative integer numbers. Thus, in order
to handle the \textquotedblleft roots\textquotedblright\ $\lambda_{i,j}$ and
their multiplicities, we are led to introduce the following notation. Let $M$
be a manifold and $M_{\Delta}(n):=\{(x_{1},\cdots,x_{n})\in M^{n}:x_{i}\neq
x_{j}$ for all $i\neq j\}$. Let $S_{n}$ denote the group of permutations of
$n$ elements and consider its action on $M_{\Delta}(n)$ given by
$(\sigma,\lambda)\mapsto\sigma\cdot\lambda=(\lambda_{\sigma(1)},\cdots
,\lambda_{\sigma(n)})$. The quotient space induced by this action is denoted
by $\operatorname*{Symm}(M_{\Delta}(n))$. Now suppose a Lie group $G$ acts on
$M$ and let $G$ act on $M_{\Delta}(n)$ in the natural way $(g,\lambda
)=(g\cdot\lambda_{1},\cdots,g\cdot\lambda_{n})$ for every $\lambda\in
M_{\Delta}(n)$. Then the actions of $G$ and $S_{n}$ on $M_{\Delta}(n)$
commute. Thus one obtains a natural action of $G$ on $\operatorname*{Symm}%
(M_{\Delta}(n))$. Given $\lambda\in M_{\Delta}(n)$, denote its equivalence
class in $\operatorname*{Symm}(M_{\Delta}(n))/G$ by $[\lambda]$. Furthermore,
if $\boldsymbol{k}=(k_{1},\cdots,k_{n})\in\mathbb{Z}_{+}^{n}$, and
$\boldsymbol{m}=(m_{1},\cdots,m_{n})\in\mathbb{Z}_{+}^{n}$, with
$0<m_{1}<\cdots<m_{n}$, then we define $\operatorname*{Symm}(M_{\Delta
}(n,\boldsymbol{k,m}))=\{\lambda=\left(  \lambda_{1},\cdots,\lambda
_{n}\right)  \in\operatorname*{Symm}(M_{\Delta}(k_{1})^{m_{1}}\times
\cdots\times\operatorname*{Symm}(M_{\Delta}(k_{n})^{m_{n}}):\lambda
_{i}=\left(  \lambda_{i,1},\cdots,\lambda_{i,k_{i}}\right)  $, $\lambda
_{i,j}=\lambda_{i^{\prime},j^{\prime}}$ iff $(i,j)=(i^{\prime},j^{\prime})\}.$

Thus, in the same spirit of \cite[Theorem A]{CamSca2018}, we obtain the
following result as a consequence of Theorem \ref{thm:main} and Lemma
\ref{lem:Lip_analytic}.

\begin{theorem}
\label{thm:analytic-moduli}The analytic moduli space of germs of
(non-homogeneous) quasihomogeneous function-germs with weights $\varpi=(p,q)$
is given, respectively, by

\begin{enumerate}
\item[(1)] $\mathbb{N}\times\frac{\operatorname*{Symm}(\mathbb{C}_{\Delta
}(n,\boldsymbol{k},\boldsymbol{m}))}{\operatorname*{Aff}(\mathbb{C})}$ if
$p=1$;

\item[(2)] $\mathbb{N}\times\mathbb{N}\times\frac{\operatorname*{Symm}%
(\mathbb{C}_{\Delta}^{\ast}(n,\boldsymbol{k},\boldsymbol{m}))}%
{\operatorname*{GL}(1,\mathbb{C})}$ if $p>1$.
\end{enumerate}
\end{theorem}

Here $\mathbb{N}$ stands for the set of non negative integer numbers
representing the exponent $m$ of $X$ in the factorization
(\ref{eq:norm_form_p=1}) and $\mathbb{N}\times\mathbb{N}$ stands for the set
of pairs of non negative integer numbers representing, respectively, the
exponents $m$ and $m_{0}$ of $X$ and $Y$ in the factorization
(\ref{eq:norm_form_p>1}).

In the sequel, we shall present some useful distinct characterizations of the
analytic moduli space of (non-homogeneous) quasihomogeneous functions.

\paragraph{\textbf{The moduli space }$\mathcal{M}_{[n,\boldsymbol{k}%
,\boldsymbol{m}],1}^{0}$}

Recall that a quasihomogeneous polynomial with weights $\varpi=(1,q)$, $q>1$,
is (uniquely) determined by (\ref{eq:norm_form_p=1}). Now let $\mathcal{M}%
_{[n,\boldsymbol{k},\boldsymbol{m}],1}^{0}$ denote the moduli space of
analytically equivalent punctured Riemann spheres with one marked puncture and
$n$ unordered weighted punctures with weight $\boldsymbol{m}=(m_{1}%
,\cdots,m_{n})$, i.e.,%

\[
\mathcal{M}_{[n,\boldsymbol{k},\boldsymbol{m}],1}^{0}=\frac
{\operatorname*{Symm}(\mathbb{C}_{\Delta}(n,\boldsymbol{k},\boldsymbol{m}%
))}{\operatorname*{Aff}(\mathbb{C})}.
\]
Then, it follows from Theorem \ref{thm:analytic-moduli} that the moduli space
of analytically equivalent quasihomogeneous polynomials with weights
$\varpi=(1,q)$, $q>1$, coincides with $\mathcal{M}_{[n,\boldsymbol{k}%
,\boldsymbol{m}],1}^{0}$. Now, following the same guidelines of \cite[Lem.
1]{CamRuas2022} and \cite[Lem. 2]{CamRuas2022}, we give a suitable description
of the space $\mathcal{M}_{[n,\boldsymbol{k},\boldsymbol{m}],1}^{0}$.

Let $\mathbb{W}_{0,\boldsymbol{k},\boldsymbol{m}}^{k-1}:=\{\lambda
\in\operatorname*{Symm}(\mathbb{C}_{\Delta}(n,\boldsymbol{k},\boldsymbol{m}%
)):\sum_{i=1}^{n}m_{i}(\lambda_{i,1}+\cdots+\lambda_{i,k_{i}})=0\}$, where
$k:=\left\Vert \boldsymbol{k}\right\Vert _{\max}=k_{1}+\cdots+k_{n}$, and
consider the natural action of $\mathbb{C}^{\ast}$ on $\mathbb{C}^{n}$ given
by $(\lambda,z)\mapsto\lambda z=(\lambda z_{1},\cdots,\lambda z_{n})$.
Clearly, this action induces an analogous action on $\mathbb{W}%
_{0,\boldsymbol{k},\boldsymbol{m}}^{k-1}$. Their cosets induce the following isomorphism.

\begin{lemma}
Let $n\in\mathbb{Z}_{+}$, $\boldsymbol{k}=(k_{1},\cdots,k_{n})\in
\mathbb{Z}_{+}^{n}$, $k:=\left\Vert \boldsymbol{k}\right\Vert _{\max}$, and
$\boldsymbol{m}=(m_{1},\cdots,m_{n})\in\mathbb{Z}_{+}^{n}$, $0<m_{1}%
<\cdots<m_{n}$. Then%
\[
\mathcal{M}_{[n,\boldsymbol{k},\boldsymbol{m}],1}^{0}\simeq\frac
{\mathbb{W}_{0,\boldsymbol{k},\boldsymbol{m}}^{k-1}}{\operatorname*{GL}%
(1,\mathbb{C})}%
\]

\end{lemma}

\begin{proof}
First notice that each class $[\lambda]\in\frac{\operatorname*{Symm}%
(\mathbb{C}_{\Delta}(n,\boldsymbol{k},\boldsymbol{m}))}{\operatorname*{Aff}%
(\mathbb{C})}$ has a representative $[\lambda^{o}]\in\mathbb{W}%
_{0,\boldsymbol{k},\boldsymbol{m}}^{k-1}$, thus it is enough to show that
$[\lambda]=[\lambda^{\prime}]\in\frac{\operatorname*{Symm}(\mathbb{C}_{\Delta
}(n,\boldsymbol{k},\boldsymbol{m}))}{\operatorname*{Aff}(\mathbb{C})}$ if and
only if $[\lambda^{o}]=[\lambda^{\prime o}]\in\frac{\mathbb{W}%
_{0,\boldsymbol{k},\boldsymbol{m}}^{k-1}}{\operatorname*{GL}(1,\mathbb{C})}$.
Let $\Phi\colon\operatorname*{Symm}(\mathbb{C}_{\Delta}(n,\boldsymbol{k}%
,\boldsymbol{m}))\rightarrow\mathbb{W}_{0,\boldsymbol{k},\boldsymbol{m}}%
^{k-1}$ be the map given by $\lambda_{i,j}\mapsto\lambda_{i,j}^{o}%
:=\lambda_{i,j}-\frac{1}{N}\sum_{r=1}^{n}\sum_{s=1}^{k_{r}}m_{r}\lambda_{r,s}%
$, where $N=\sum_{i=1}^{n}m_{i}\cdot k_{i}$. Then a straightforward
calculation shows that $\Phi(a\lambda+b)=a\Phi(\lambda)$, i.e., if
$\lambda^{\prime}=a\lambda+b$ then $\lambda_{0}^{\prime}=a\lambda_{0}$. Thus,
$\Phi$ induces a well defined map $\overline{\Phi}\colon\frac
{\operatorname*{Symm}(\mathbb{C}_{\Delta}(n,\boldsymbol{k},\boldsymbol{m}%
))}{\operatorname*{Aff}(\mathbb{C})}\rightarrow\frac{\mathbb{W}%
_{0,\boldsymbol{k},\boldsymbol{m}}^{k-1}}{\operatorname*{GL}(1,\mathbb{C})}$.
This map is clearly surjective, thus it suffices to show that it is injective.
Now let $\lambda,\lambda^{\prime}\in\operatorname*{Symm}(\mathbb{C}_{\Delta
}(n,\boldsymbol{k},\boldsymbol{m}))$ be such that $\Phi(\lambda^{\prime
})=a\Phi(\lambda)$, this implies the existence of $b,b^{\prime}\in\mathbb{C}$
such that $\lambda_{i,j}^{\prime}-b^{\prime}=a(\lambda_{i,j}-b)$. The result
then follows.
\end{proof}

This is equivalent to say that the polynomial $Q_{\lambda}$ may be written in
the normal form $Q_{\lambda}(t)=t^{N}+a_{N-2}t^{N-2}+\cdots+a_{1}t+a_{0}$. In
other words, any quasihomogeneous polynomial $F$ with weights $\varpi=(1,q)$,
$q>1$, may be reduced to the normal form%
\begin{align*}
F(X,Y) &  =x^{m+Nq}Q_{\lambda}(Y/X^{q})\\
&  =X^{m+Nq}(\frac{Y^{N}}{X^{qN}}+a_{n-2}\frac{Y^{N-2}}{X^{q(N-2)}}%
+\cdots+a_{1}\frac{Y}{X^{q}}+a_{0})\\
&  =X^{m}(Y^{N}+a_{N-2}Y^{N-2}X^{2q}+\cdots+a_{1}YX^{(N-1)q}+a_{0}X^{Nq}).
\end{align*}
Now we have to split our study in two distinct cases. (1) In order to simplify
our arguing, let us introduce another notation. Let $\mathbb{W}%
_{0,\boldsymbol{k},\boldsymbol{m}}^{k-1,\ast}:=\{\lambda\in$ $\mathbb{W}%
_{0,\boldsymbol{k},\boldsymbol{m}}^{k-1}:\lambda_{ij}\neq0$ $\forall
i=1,\ldots,n$, $j=1,\ldots,k_{i}\}$, then the above result leads to the
following definition:%
\[
\mathcal{M}_{[n,\boldsymbol{k},\boldsymbol{m}],1}^{0,\ast}:=\frac
{\mathbb{W}_{0,\boldsymbol{k},\boldsymbol{m}}^{k-1,\ast}}{\operatorname*{GL}%
(1,\mathbb{C})}.
\]

Now consider the above defined actions of $\mathbb{Z}_{n}$ on $\mathbb{C}^{m}$
and $\operatorname*{Symm}(\mathbb{C}^{m})$, and let $\mathbb{W}%
_{0,n,\boldsymbol{k},\boldsymbol{m}}^{k-2,\ast}\subset\mathbb{W}%
_{0,\boldsymbol{k},\boldsymbol{m}}^{k-1,\ast}$ be given by $\mathbb{W}%
_{0,n,\boldsymbol{k},\boldsymbol{m}}^{k-2,\ast}=\{\lambda=(\lambda_{1}%
,\cdots,\lambda_{n})\in\mathbb{W}_{0,\boldsymbol{k},\boldsymbol{m}}^{k-1,\ast
}:\prod_{i=1}^{n}\prod_{j=1}^{k_{i}}\lambda_{i,j}^{m_{i}}=1\}$, then we can
reduce the open subset $\mathcal{M}_{[n,\boldsymbol{k},\boldsymbol{m}%
],1}^{0,\ast}\subset\mathcal{M}_{[n,\boldsymbol{k},\boldsymbol{m}],1}^{0}$ a
bit further.

\begin{lemma}
\label{Kang1}Let $n\in\mathbb{Z}_{+}$, $\boldsymbol{k}=(k_{1},\cdots,k_{n}%
)\in\mathbb{Z}_{+}^{n}$, $k:=\left\Vert \boldsymbol{k}\right\Vert _{\max}$,
$\boldsymbol{m}=(m_{1},\cdots,m_{n})\in\mathbb{Z}_{+}^{n}$, $0<m_{1}%
<\cdots<m_{n}$, and $N=m_{1}k_{1}+\cdots+m_{n}k_{n}$. Then%
\[
\mathcal{M}_{[n,\boldsymbol{k},\boldsymbol{m}],1}^{0,\ast}=\frac
{\mathbb{W}_{0,\boldsymbol{k},\boldsymbol{m}}^{k-1,\ast}}{\operatorname*{GL}%
(1,\mathbb{C})}\simeq\frac{\mathbb{W}_{0,n,\boldsymbol{k},\boldsymbol{m}%
}^{k-2,\ast}}{\mathbb{Z}_{N}}.
\]

\end{lemma}

\begin{proof}
Consider the map $\Psi\colon\mathbb{W}_{0,\boldsymbol{k},\boldsymbol{m}%
}^{k-1,\ast}\rightarrow\mathbb{W}_{0,n,\boldsymbol{k},\boldsymbol{m}%
}^{k-2,\ast}$ given by $\lambda=(\lambda_{1},\cdots,\lambda_{n})\mapsto
\overline{\lambda}=\frac{1}{\prod_{i=1}^{n}\prod_{j=1}^{k_{i}}\lambda
_{i,j}^{m_{i}/N}}\lambda$. Similarly, it suffices to show that $[\lambda
]=[\lambda^{\prime}]\in\frac{\mathbb{W}_{0,\boldsymbol{k},\boldsymbol{m}%
}^{k-1,\ast}}{\operatorname*{GL}(1,\mathbb{C})}$ if and only if $\left[
\overline{\lambda}\right]  =\left[  \overline{\lambda^{\prime}}\right]
\in\frac{\mathbb{W}_{0,n,\boldsymbol{k},\boldsymbol{m}}^{k-2,\ast}}%
{\mathbb{Z}_{n}}$. If $\lambda^{\prime}=a\lambda$, then%
\[
\overline{\lambda^{\prime}}=\frac{1}{\prod_{i=1}^{n}\prod_{j=1}^{k_{i}%
}(\lambda_{i,j}^{\prime})^{m_{i}/N}}\cdot\lambda^{\prime}=\frac{a}{\omega
^{-s}\cdot a}\frac{1}{\prod_{i=1}^{n}\prod_{j=1}^{k_{i}}\lambda_{i,j}%
^{m_{i}/N}}\lambda=\omega^{s}\cdot\overline{\lambda},
\]
where $\omega=\exp(2\pi\boldsymbol{i}/N)$ and $s\in\mathbb{Z}$. Thus, $\Psi$
induces a map $\overline{\Psi}\colon\frac{\mathbb{W}_{0,\boldsymbol{k}%
,\boldsymbol{m}}^{k-1,\ast}}{\operatorname*{GL}(1,\mathbb{C})}\rightarrow
\frac{\mathbb{W}_{0,n,\boldsymbol{k},\boldsymbol{m}}^{k-2,\ast}}%
{\mathbb{Z}_{n}}$. Clearly $\overline{\Psi}$ is surjective, so it is enough to
prove that it is injective. But notice that $\overline{\lambda^{\prime}%
}=\omega^{k}\overline{\lambda}$ implies
\[
\frac{\lambda^{\prime}}{\prod_{i=1}^{n}\prod_{j=1}^{k_{i}}(\lambda
_{i,j}^{\prime})^{m_{i}/N}}=\frac{\omega^{k}\lambda}{\prod_{i=1}^{n}%
\prod_{j=1}^{k_{i}}\lambda_{i,j}^{m_{i}/N}}\text{.}%
\]
But this is equivalent to say that $\lambda^{\prime}=a\lambda$ for some
$a\in\mathbb{C}$.
\end{proof}

Notice that the polynomial $Q_{\lambda}$ may be written in the normal form
\begin{equation}
Q_{\lambda}(t)=t^{N}+a_{N-2}t^{N-2}+\cdots+a_{1}t+(-1)^{N}%
\label{eq:pol_Q-Lambda_gen}%
\end{equation}
for any $\lambda\in\mathbb{W}_{0,\boldsymbol{k},\boldsymbol{m}}^{k-1,\ast}$.
In other words, any quasihomogeneous polynomial $F$ with weights
$\varpi=(1,q)$, $q>1$, may be reduced to the normal form%
\begin{align*}
F(X,Y) &  =x^{m+Nq}Q_{\lambda}(Y/X^{q})\\
&  =X^{m+Nq}(\frac{Y^{N}}{X^{qN}}+a_{n-2}\frac{Y^{N-2}}{X^{q(N-2)}}%
+\cdots+a_{1}\frac{Y}{X^{q}}+(-1)^{N})\\
&  =X^{m}(Y^{N}+a_{N-2}Y^{N-2}X^{2q}+\cdots+a_{1}YX^{(N-1)q}+(-1)^{N}X^{Nq})
\end{align*}
for any $\lambda\in\mathbb{W}_{0,\boldsymbol{k},\boldsymbol{m}}^{k-1,\ast}$.

\begin{remark}
Since the action of \ $\mathbb{Z}_{N}$ on $\mathbb{W}_{0,n,\boldsymbol{k}%
,\boldsymbol{m}}^{k-2,\ast}$ is properly discontinuous, the dimension of the
moduli space $\mathcal{M}_{[n,\boldsymbol{k},\boldsymbol{m}],1}^{0,\ast}$
corresponds to the number of free coeficients of the polynomial
(\ref{eq:pol_Q-Lambda_gen}). Thus $\mathcal{M}_{[n,\boldsymbol{k}%
,\boldsymbol{m}],1}^{0,\ast}$ is a complex manifold with dimension $N-2$.
\end{remark}

As for the second case, let us introduce a bit more notation: Let%
\[
\operatorname*{Sing}(\mathcal{M}_{[n,\boldsymbol{k},\boldsymbol{m}],1}%
^{0}):=\mathcal{M}_{[n,\boldsymbol{k},\boldsymbol{m}],1}^{0}\setminus
\mathcal{M}_{[n,\boldsymbol{k},\boldsymbol{m}],1}^{0,\ast}.
\]
Notice that $\operatorname*{Sing}(\mathcal{M}_{[n,\boldsymbol{k}%
,\boldsymbol{m}],1}^{0})$ decomposes in terms of irreducible components as
\[
\operatorname*{Sing}(\mathcal{M}_{[n,\boldsymbol{k},\boldsymbol{m}],1}%
^{0})=\cup_{\iota=1}^{n}\mathcal{M}_{[n,\boldsymbol{k},\boldsymbol{m}%
],1}^{0,\iota},
\]
where $\mathbb{W}_{0,\boldsymbol{k},\boldsymbol{m}}^{k-1,\iota}=\{\lambda
\in\mathbb{W}_{0,\boldsymbol{k},\boldsymbol{m}}^{k-1}:\lambda_{\iota,\tau}=0$
for some $\tau=1,\ldots,k_{\iota}\}$ and $\mathcal{M}_{[n,\boldsymbol{k}%
,\boldsymbol{m}],1}^{0,\iota}=\left.  \mathbb{W}_{0,\boldsymbol{k}%
,\boldsymbol{m}}^{k-1,\iota}\right/  \operatorname*{GL}(1,\mathbb{C})$. On the
other hand, let
\[
\mathbb{W}_{0,1,\boldsymbol{k},\boldsymbol{m}}^{k-3}:=\{\lambda\in
\mathbb{W}_{0,\boldsymbol{k},\boldsymbol{m}}^{k-1}\setminus\mathbb{W}%
_{0,\boldsymbol{k},\boldsymbol{m}}^{k-1,\ast}:\sum_{i=1}^{n}\sum_{j=1}^{k_{i}%
}\lambda_{1,1}^{m_{1}}\cdots\lambda_{i,j}^{m_{i}-1}\cdots\lambda_{n,k_{n}%
}^{m_{n}}=1\}.
\]
\ Notice that $\mathbb{W}_{0,1,\boldsymbol{k},\boldsymbol{m}}^{k-3}$
decomposes in terms of irreducible components as%
\[
\mathbb{W}_{0,1,\boldsymbol{k},\boldsymbol{m}}^{k-3}=\bigcup_{i=1}%
^{n}\mathbb{W}_{0,1,\boldsymbol{k},\boldsymbol{m}}^{k-3,\iota},
\]
$\mathbb{W}_{0,1,\boldsymbol{k},\boldsymbol{m}}^{k-3,\iota}:=\{\lambda
\in\mathbb{W}_{0,\boldsymbol{k},\boldsymbol{m}}^{k-1}\setminus\mathbb{W}%
_{0,\boldsymbol{k},\boldsymbol{m}}^{k-1,\ast}:\sum_{i=1}^{n}\sum_{j=1}^{k_{i}%
}\lambda_{1,1}^{m_{1}}\cdots\lambda_{i,j}^{m_{i}-1}\cdots\lambda_{n,k_{n}%
}^{m_{n}}=1$ and $\lambda_{\iota,\tau}=0$ for some $\tau=1,\ldots,k_{\iota}\}$.

The next result describes the components of $\operatorname*{Sing}%
(\mathcal{M}_{[n,\boldsymbol{k},\boldsymbol{m}],1}^{0})$ in more detail.

\begin{lemma}
Let $n\in\mathbb{Z}_{+}$, $\boldsymbol{k}=(k_{1},\cdots,k_{n})\in
\mathbb{Z}_{+}^{n}$, $k:=\left\Vert \boldsymbol{k}\right\Vert _{\max}$,
$\boldsymbol{m}=(m_{1},\cdots,m_{n})\in\mathbb{Z}_{+}^{n}$, $0<m_{1}%
<\cdots<m_{n}$ and $N=k_{1}m_{1}+\cdots+k_{n}m_{n}$. Then%
\[
\mathcal{M}_{[n,\boldsymbol{k},\boldsymbol{m}],1}^{0,\iota}=\frac
{\mathbb{W}_{0,\boldsymbol{k},\boldsymbol{m}}^{k-1,\iota}}{\operatorname*{GL}%
(1,\mathbb{C})}\simeq\frac{\mathbb{W}_{0,1,\boldsymbol{k},\boldsymbol{m}%
}^{k-3,\iota}}{\mathbb{Z}_{N-1}}.
\]

\end{lemma}

\begin{proof}
Let $\Xi\colon\mathbb{W}_{0,\boldsymbol{k},\boldsymbol{m}}^{k-1,\iota
}\rightarrow\mathbb{W}_{0,1,\boldsymbol{k},\boldsymbol{m}}^{k-3,\iota}$ be the
map defined by%
\[
\lambda\mapsto\overline{\lambda}=\frac{1}{\left(  \sum_{i=1}^{n}\sum
_{j=1}^{k_{\iota}}\lambda_{1,1}^{m_{1}}\cdots\lambda_{i,j}^{m_{i}-1}%
\cdots\lambda_{n,k_{n}}^{m_{n}}\right)  ^{\frac{1}{N-1}}}\lambda.
\]
Analogously, it suffices to prove that $[\lambda]=[\lambda^{\prime}]\in
\frac{\mathbb{W}_{0,\boldsymbol{k},\boldsymbol{m}}^{k-1,\iota}}%
{\operatorname*{GL}(1,\mathbb{C})}$ if and only if $\left[  \overline{\lambda
}\right]  =\left[  \overline{\lambda^{\prime}}\right]  \in\frac{\mathbb{W}%
_{0,1,\boldsymbol{k},\boldsymbol{m}}^{k-3,\iota}}{\mathbb{Z}_{N-1}}$. If
$\lambda^{\prime}=a\lambda$, then%
\begin{align*}
\overline{\lambda^{\prime}}  &  =\frac{1}{\left(  \sum_{i=1}^{n}\sum
_{j=1}^{k_{\iota}}\lambda_{1,1}^{m_{1}}\cdots\lambda_{i,j}^{m_{i}-1}%
\cdots\lambda_{n,k_{n}}^{m_{n}}\right)  ^{\frac{1}{N-1}}}\cdot\lambda^{\prime
}\\
&  =\frac{a}{\omega^{-s}\cdot a}\frac{1}{\left(  \sum_{i=1}^{n}\sum
_{j=1}^{k_{\iota}}\lambda_{1,1}^{m_{1}}\cdots\lambda_{i,j}^{m_{i}-1}%
\cdots\lambda_{n,k_{n}}^{m_{n}}\right)  ^{\frac{1}{N-1}}}\lambda=\omega
^{s}\cdot\overline{\lambda},
\end{align*}
where $\omega=\exp(2\pi\boldsymbol{i}/(N-1))$ and $s\in\mathbb{Z}$. Thus,
$\Xi$ induces a map $\overline{\Xi}\colon\frac{\mathbb{W}_{0,\boldsymbol{k}%
,\boldsymbol{m}}^{k-1,\iota}}{\operatorname*{GL}(1,\mathbb{C})}\rightarrow
\frac{\mathbb{W}_{0,1,\boldsymbol{k},\boldsymbol{m}}^{k-3,\iota}}%
{\mathbb{Z}_{N-1}}$. Clearly, $\overline{\Xi}$ is surjective; so, it is enough
to prove that it is injective. But notice that $\overline{\lambda^{\prime}%
}=\omega^{k}\overline{\lambda}$ implies that
\[
\lambda^{\prime}=\left(  \frac{\sum_{i=1}^{n}\sum_{j=1}^{k_{\iota}}\left(
\lambda_{1,1}^{\prime}\right)  ^{m_{1}}\cdots(\lambda_{i,j}^{\prime}%
)^{m_{i}-1}\cdots(\lambda_{n,k_{n}}^{\prime})^{m_{n}}}{\sum_{i=1}^{n}%
\sum_{j=1}^{k_{\iota}}\lambda_{1,1}^{m_{1}}\cdots\lambda_{i,j}^{m_{i}-1}%
\cdots\lambda_{n,k_{n}}^{m_{n}}}\right)  ^{\frac{1}{N-1}}\cdot\omega
^{k}\lambda\text{.}%
\]

\end{proof}

Since the action of \ $\mathbb{Z}_{N-1}$ on $\mathbb{W}_{0,1,\boldsymbol{k}%
,\boldsymbol{m}}^{k-3,i}$ is properly discontinuous, we conclude that%
\begin{equation}
\operatorname*{Sing}(\mathcal{M}_{[n,\boldsymbol{k},\boldsymbol{m}],1}%
^{0})\simeq\bigcup_{i=1}^{n}\frac{\mathbb{W}_{0,1,\boldsymbol{k}%
,\boldsymbol{m}}^{k-3,i}}{\mathbb{Z}_{N-1}}\label{eq:strat}%
\end{equation}
is a disjoint union of complex manifolds.

The above result is equivalent to say that the polynomial $Q_{\lambda}$ may be
written in the normal form
\begin{equation}
Q_{\lambda}(t)=t^{N}+a_{N-2}t^{N-2}+\cdots+a_{m_{i}+1}t^{m_{i}+1}%
+(-1)^{N-m_{i}}t^{m_{i}} \label{eq:pol_Q-Lambda_sing}%
\end{equation}
for any $\lambda\in\mathbb{W}_{0,\boldsymbol{k},\boldsymbol{m}}^{k-1,i}$. In
other words, any quasihomogeneous polynomial $F(X,Y)$ with weights
$\varpi=(1,q)$, $q>1$, may be reduced to the normal form%
\begin{align*}
F  &  =X^{m+Nq}Q_{\lambda}(Y/X^{q})\\
&  =X^{m+Nq}(\frac{Y^{N}}{X^{qN}}+a_{N-2}\frac{Y^{N-2}}{X^{q(N-2)}}%
+\cdots+a_{m_{i}+1}\frac{Y^{m_{i}+1}}{X^{q(m_{i}+1)}}+(-1)^{N-m_{i}}%
\frac{Y^{m_{i}}}{X^{qm_{i}}})\\
&  =X^{m}(Y^{N}+a_{N-2}Y^{N-2}X^{2q}+\cdots+a_{m_{i}+1}Y^{m_{i}+1}%
X^{(m_{i}+1)q}+(-1)^{N-m_{i}}Y^{m_{i}}X^{qm_{i}})
\end{align*}
for any $\lambda\in\mathbb{W}_{0,\boldsymbol{k},\boldsymbol{m}}^{k-1,i}$,
$i=1,\ldots,n.$

\begin{remark}
The subspace $\operatorname*{Sing}(\mathcal{M}_{[n,\boldsymbol{k}%
,\boldsymbol{m}],1}^{0})$ is a stratification given by (\ref{eq:strat}).
Furthermore, since the action of \ $\mathbb{Z}_{N-1}$ on $\mathbb{W}%
_{0,1,\boldsymbol{k},\boldsymbol{m}}^{k-3,i}$ is properly discontinuous, the
dimension of each stratum $\left.  \mathbb{W}_{0,1,\boldsymbol{k}%
,\boldsymbol{m}}^{k-3,i}\right/  $ $\mathbb{Z}_{N-1}$ corresponds to number of
free coefficients of the polynomial (\ref{eq:pol_Q-Lambda_sing}). Thus,
$\operatorname*{Sing}(\mathcal{M}_{[n,\boldsymbol{k},\boldsymbol{m}],1}^{0})$
is a stratification by the complex manifolds $\left.  \mathbb{W}%
_{0,1,\boldsymbol{k},\boldsymbol{m}}^{k-3,i}\right/  $ $\mathbb{Z}_{N-1}$ with
dimension $N-2-m_{i}$. Therefore, the set $\operatorname*{Sing}(\mathcal{M}%
_{[n,\boldsymbol{k},\boldsymbol{m}],1}^{0})$ corresponds to the singular part
of the moduli space $\mathcal{M}_{[n,\boldsymbol{k},\boldsymbol{m}],1}^{0}$.
\end{remark}

\paragraph{\textbf{The moduli space }$\mathcal{M}_{[n,\boldsymbol{k}%
,\boldsymbol{m}],2}^{0}$}

If $p>1$, a quasihomogeneous polynomial with weights $\varpi=(p,q)$ is
(uniquely) determined by (\ref{eq:norm_form_p>1}). Now let $\mathcal{M}%
_{[n,\boldsymbol{k},\boldsymbol{m}],2}^{0}$ denote the moduli space of
analytically equivalent punctured Riemann spheres with two marked puncture and
$n$ unordered weighted punctures with weight $\boldsymbol{m}=(m_{1}%
,\cdots,m_{n})$, i.e.,%

\[
\mathcal{M}_{[n,\boldsymbol{k},\boldsymbol{m}],2}^{0}=\frac
{\operatorname*{Symm}(\mathbb{C}_{\Delta}^{\mathbb{\ast}}(n,\boldsymbol{k}%
,\boldsymbol{m}))}{\operatorname*{GL}(1,\mathbb{C})}.
\]
Then, it follows from Theorem \ref{thm:analytic-moduli} that the moduli space
of analytically equivalent quasihomogeneous polynomials with weights
$\varpi=(p,q)$, $1<p<q$, coincides with $\mathcal{M}_{[n,\boldsymbol{k}%
,\boldsymbol{m}],2}^{0}$. Now, following the same guidelines of \cite[Lem.
4]{CamRuas2022}, we give a suitable description of the space $\mathcal{M}%
_{[n,\boldsymbol{k},\boldsymbol{m}],2}^{0}$.

Let $\mathbb{V}_{n,\boldsymbol{k},\boldsymbol{m}}^{k-1}:=\{\lambda
\in\operatorname*{Symm}(\mathbb{C}_{\Delta}^{\mathbb{\ast}}(n,\boldsymbol{k}%
,\boldsymbol{m})):\prod_{i=1}^{n}\prod_{j=1}^{k_{i}}\lambda_{i,j}^{m_{i}}%
=1\}$, $k=\left\Vert \boldsymbol{k}\right\Vert _{\max}$, and $N=\sum_{i=1}%
^{n}m_{i}\cdot k_{i}$. On $\mathbb{V}_{n,\boldsymbol{k},\boldsymbol{m}}^{k-1}%
$, we may consider the (effective) action of $\mathbb{Z}_{N}$ given by the
multiplication by the $N$-th roots of unity, i.e., $\varsigma^{s}%
\mapsto\varsigma^{s}\cdot\lambda$ with $\varsigma=\exp(\frac{2\pi
\boldsymbol{i}}{N})$. Therefore, we have the following isomorphism.

\begin{lemma}
\label{Kang2}Let $n\in\mathbb{Z}_{+}$, $\boldsymbol{k}=(k_{1},\cdots,k_{n}%
)\in\mathbb{Z}_{+}^{n}$, $k=\left\Vert \boldsymbol{k}\right\Vert _{\max}$,
$N=m_{1}k_{1}+\cdots+m_{n}k_{n}$, and $\boldsymbol{m}=(m_{1},\cdots,m_{n}%
)\in\mathbb{Z}_{+}^{n}$, $0<m_{1}<\cdots<m_{n}$. Then%
\begin{equation}
\mathcal{M}_{[n,\boldsymbol{k},\boldsymbol{m}],2}^{0}\simeq\frac
{\mathbb{V}_{n,\boldsymbol{k},\boldsymbol{m}}^{k-1}}{\mathbb{Z}_{N}}.
\label{eq23}%
\end{equation}

\end{lemma}

\begin{proof}
Consider the map $\Psi\colon\operatorname*{Symm}(\mathbb{C}_{\Delta
}^{\mathbb{\ast}}(n,\boldsymbol{k},\boldsymbol{m}))\rightarrow\mathbb{V}%
_{n,\boldsymbol{k},\boldsymbol{m}}^{k-1}$ given by $\lambda=(\lambda
_{1},\cdots,\lambda_{n})\mapsto\overline{\lambda}=\frac{1}{\prod_{i=1}%
^{n}\prod_{j=1}^{k_{i}}\lambda_{i,j}^{m_{i}/N}}\lambda$. As before, it
suffices to prove that $[\lambda]=[\lambda^{\prime}]\in\frac
{\operatorname*{Symm}(\mathbb{C}_{\Delta}^{\mathbb{\ast}}(n,\boldsymbol{k}%
,\boldsymbol{m}))}{\operatorname*{GL}(1,\mathbb{C})}$ if and only if $\left[
\overline{\lambda}\right]  =\left[  \overline{\lambda^{\prime}}\right]
\in\frac{\mathbb{V}_{n,\boldsymbol{k},\boldsymbol{m}}^{k-1}}{\mathbb{Z}_{N}}$.
If $\lambda^{\prime}=a\lambda$, then%
\[
\overline{\lambda^{\prime}}=\frac{1}{\prod_{i=1}^{n}\prod_{j=1}^{k_{i}%
}(\lambda_{i,j}^{\prime})^{m_{i}/N}}\cdot\lambda^{\prime}=\frac{a}{\omega
^{-s}\cdot a}\frac{1}{\prod_{i=1}^{n}\prod_{j=1}^{k_{i}}\lambda_{i,j}%
^{m_{i}/N}}\lambda=\omega^{s}\cdot\overline{\lambda},
\]
where $\omega=e^{\frac{2\pi\boldsymbol{i}}{N}}$ and $s\in\mathbb{Z}$. Thus,
$\Psi$ induces a map $\overline{\Psi}\colon\frac{\operatorname*{Symm}%
(\mathbb{C}_{\Delta}^{\mathbb{\ast}}(n,\boldsymbol{k},\boldsymbol{m}%
))}{\operatorname*{GL}(1,\mathbb{C})}\rightarrow\frac{\mathbb{V}%
_{n,\boldsymbol{k},\boldsymbol{m}}^{k-1}}{\mathbb{Z}_{N}}$. Clearly
$\overline{\Psi}$ is surjective, so it suffices to prove that it is injective.
But notice that $\overline{\lambda^{\prime}}=\omega^{k}\overline{\lambda}$
implies
\[
\frac{\lambda^{\prime}}{\prod_{i=1}^{n}\prod_{j=1}^{k_{i}}(\lambda
_{i,j}^{\prime})^{m_{i}/N}}=\frac{\omega^{k}\lambda}{\prod_{i=1}^{n}%
\prod_{j=1}^{k_{i}}\lambda_{i,j}^{m_{i}/N}}\text{.}%
\]
Therefore, $\lambda^{\prime}=a\lambda$ for some $a\in\mathbb{C}$.
\end{proof}

The last result is equivalent to say that the polynomial $Q_{\lambda}$ may be
written in the normal form
\begin{equation}
Q_{\lambda}(t)=t^{N}+a_{N-1}t^{N-1}+\cdots+a_{1}t+(-1)^{N}
\label{eq:pol_Q-Lambda_II}%
\end{equation}
for any $\lambda\in\mathbb{V}_{n,\boldsymbol{k},\boldsymbol{m}}^{k-1}$. In
other words, any quasihomogeneous polynomial $F$ with weights $\varpi=(p,q)$,
$1<p<q$, may be reduced to the normal form%
\begin{align*}
F(X,Y)  &  =X^{m+Nq}Y^{m_{0}}Q_{\lambda}(Y^{p}/X^{q})\\
&  =X^{m+Nq}Y^{m_{0}}(\frac{Y^{N}}{X^{qN}}+a_{N-1}\frac{Y^{N-1}}{X^{q(N-1)}%
}+\cdots+a_{1}\frac{Y}{X^{q}}+(-1)^{N})\\
&  =X^{m}Y^{m_{0}}(Y^{N}+a_{N-1}Y^{N-1}X^{2q}+\cdots+a_{1}YX^{(N-1)q}%
+(-1)^{N}X^{Nq})
\end{align*}
for any $\lambda\in\mathbb{V}_{n,\boldsymbol{k},\boldsymbol{m}}^{k-1}$.

\begin{remark}
Since the action of \ $\mathbb{Z}_{N}$ on $\mathbb{V}_{n,\boldsymbol{k}%
,\boldsymbol{m}}^{k-1}$ is properly discontinuous, the above result says that
the dimension of the moduli space $\mathcal{M}_{[n,\boldsymbol{k}%
,\boldsymbol{m}],2}^{0}$ corresponds to the number of free coeficients of the
polynomial (\ref{eq:pol_Q-Lambda_II}), i.e., $\dim(\mathcal{M}%
_{[n,\boldsymbol{k},\boldsymbol{m}],1}^{0,\ast})=N-1$.
\end{remark}

\paragraph{\textbf{A\ final remark on Henry-Parusi\'{n}ski invariants}.}

We conclude this paper with a final remark on the Lipschitz invariants of
complex commode quasihomogeneous polynomials. As shown in \cite{CamRuas2022},
the polynomials $f_{1}(x,y)=y^{3p}-\frac{33}{9}y^{2p}x^{q}+\frac{24}{9}%
y^{p}x^{2q}-x^{3q}$, $f_{2}(x,y)=y^{3p}-\frac{6}{9}y^{2p}x^{q}-\frac{15}%
{9}y^{p}x^{2q}-x^{3q}$, and $f_{3}(x,y)=y^{3p}+\frac{39}{9}y^{2p}x^{q}%
+\frac{40}{9}y^{p}x^{2q}-x^{3q}$ share the same Henry-Parusi\'{n}ski
invariants but are not analytically equivalent. This raises the question of
whether another Lipschitz invariant exists for these polynomials, or if they
are in fact bi-Lipschitz equivalent.

In a joint discussion with the first named author and Maria Aparecida Ruas at
ICMC-USP in August 2017, just after the Bedlevo meeting on Lipschitz geometry,
Terry Gaffney noted that, for commode quasihomogeneous polynomials, the
Henry-Parusi\'{n}ski invariants are essentially the critical values of the
polynomial $Q_{\lambda}$. What this paper demonstrates is that this
information alone is not sufficient to determine the bi-Lipschitz type of a
commode quasihomogeneous function-germ; instead, both the critical values and
their multiplicities are needed (cf. Lemmas \ref{Kang1} and \ref{Kang2}).

\subsection*{Acknowledgments}

The authors would like to thank the anonymous referee for useful comments and
suggestions that improved the presentation of this paper, as well as Edson
Sampaio for valuable discussions and suggestions.


\begin{thebibliography}{99}                                                                                               %


\bibitem {Al}S. Alvarez. \emph{Semialgebraic Lipschitz equivalence of
polynomial functions}. Res Math Sci 12, 76 (2025). https://doi.org/10.1007/s40687-025-00558-y

\bibitem {BF2000}L. Birbrair and A. Fernandes. \emph{Metric theory of
semialgebraic curves}. Revista Matemática Complutense, 13 (2) 369, (2000).

\bibitem {BFP}L. Birbrair, A. Fernandes and D. Panazzolo. \emph{Lipschitz
classification of functions on Holder triangle.} St. Petersburg Math. J. v.
20, n. 5, 681--686, (2009).

\bibitem {Bu}W. Burau. \emph{Kennzeichung der Schlauchknoten.} Abh. Math. Sem.
Hamburg, 9 (1932), 125--133.

\bibitem {CamSca2018}L. Câmara and B. Scárdua. \emph{A comprehensive approach
to the moduli space of quasihomogeneous singularities}. (English summary)
Singularities and foliations. geometry, topology and applications, 459--487,
Springer Proc. Math. Stat., 222, Springer, Cham, (2018).

\bibitem {CamRuas2022}L. M. Câmara and M. A. S. Ruas. \emph{On the moduli
space of quasihomogeneous functions}. Bull Braz Math Soc, New Series 53,
895--908, (2022).

\bibitem {FR}A. Fernandes and M. Ruas. \emph{Bilipschitz determinacy of
quasihomogeneous germs}. Glasgow Math. J. 46, pp. 77--82, (2004).

\bibitem {FR2}A. Fernandes and M. Ruas. \emph{Rigidity of bi-Lipschitz
equivalence of weighted homogeneous function-germs in the plane.} Proc. Amer.
Math. Soc. 141, 1125--1133, (2013).

\bibitem {Fu}T. Fukuda. \emph{Types topologiques des polynômes}. Publications
mathématiques de l'I.H.É.S., tome 46 (1976), p. 87--106

\bibitem {HP1}J.-P. Henry and A. Parusi\'{n}ski. \emph{Existence of moduli for
bi-Lipschitz equivalence of analytic functions.} Compositio Math. 136,
217-235, (2003).

\bibitem {HP2}J.-P. Henry and A. Parusi\'{n}ski. \emph{Invariants of
bi-Lipschitz equivalence of real analytic functions.} in Geometric Singularity
Theory, Banach Center Publications 65, PWN, Warszawa, 67-75, (2004).

\bibitem {K}T.-C. Kuo. \emph{On $C^{0}$-sufficiency of jets of potential
functions}. Topology 8, 167--171, (1969).

\bibitem {KP}S. Koike and A. Parusi\'{n}ski. \emph{Equivalence relations for
two variable real analytic function germs}. J. Math. Soc. Japan Volume 65, N
1, 237--276, (2013).

\bibitem {P}A. Parusi\'{n}ski. \emph{A criterion for topological equivalence
of two variable complex analytic function germs}. Proc. Japan Acad. Ser. A
Math. Sci. Volume 84, N 8, 147--150, (2008).

\bibitem {Re}J. E. Reeve. \emph{A summary of results in the topological
classification of plane algebroid singularities}. Università e Politecnico di
Torino, Rendiconti del Seminaro Matematico 1954/55.

\bibitem {RT}J.-J. Risler and D. Trotman. \emph{Bilipschitz invariance of the
multiplicity}. Bull. London Math. Soc. 29, 200-204, (1997).

\bibitem {S}J. Edson Sampaio. \emph{Bi-Lipschitz homeomorphic subanalytic sets
have bi-Lipschitz homeomorphic tangent cones}. Selecta Mathematica: New
Series, vol. 22, no. 2, 553-559, (2016).

\bibitem {Wh}H. Whitney. \emph{Local properties of analytic varieties}.
Differential and Combinatorial Topology. A Symposium in honor of Marston
Morse. Princeton Univ. Press (1965).

\bibitem {Za}O. Zariski. \emph{On the Topology of Algebroid Singularities}.
Amer. J. Math. 54 (1932), no. 3, 453?465.

%\bibitem {TeseSergio-2021}Correia, Sergio Alvarez Araujo. \emph{Semialgebraic Lipschitz equivalence of polynomial functions}. 2021. 99 f. Tese (Doutorado em Matemï¿½tica) -- Centro de Ciï¿½ncias, Universidade Federal do Cearï¿½, Fortaleza, 2021. https://repositorio.ufc.br/handle/riufc/64519.

\end{thebibliography}
\end{document}